\documentclass[letterpaper, 10pt, tbtags, conference]{ieeeconf}
\IEEEoverridecommandlockouts

\usepackage{mathtools}

\usepackage{mathrsfs}
\usepackage{dsfont}
\usepackage{amsfonts,amssymb}
\usepackage{cite}

\usepackage[standard]{ntheorem}
\usepackage{graphicx}

\newcommand{\PR}{\mathds{P}}
\newcommand{\EXP}{\mathds{E}}

\newcommand\VEC{\mathbf}

\newcommand\MEAN[1]{\langle \VEC #1\rangle}
\newcommand\QUAD[2]{ \mathinner{#1}^\TRANS #2 {#1} }

\newcommand\TRANS{\intercal}
\DeclareMathOperator\VVEC{vec}

\newcommand\reals{\mathds{R}}

\newtheorem{problem}        {Problem}

\newcommand\important[1]{\textbf{\itshape #1}}

\begin{document}

\title{\LARGE \bf Team-Optimal  Solution of Finite Number of Mean-Field Coupled LQG Subsystems}

\author{Jalal Arabneydi and Aditya Mahajan
\thanks{This work was supported by the
Natural Sciences and Engineering Research Council of Canada through Grant NSERC-RGPIN 402753-11.}
\thanks{Jalal Arabneydi and Aditya Mahajan are with the Department of
  Electrical and Computer Engineering, McGill University, 
  Montreal, QC, Canada. 
        Email: $\{$jalal.arabneydi@mail.mcgill.ca,
        aditya.mahajan@mcgill.ca$\}$}
}

\maketitle

\vspace*{-5.2cm}{\footnotesize{Proceedings of IEEE Conference on Decision and Control, 2015.}}
\vspace*{4.45cm}

\begin{abstract}
  A decentralized control system with linear dynamics, quadratic cost, and
  Gaussian disturbances is considered. The system consists of a finite number
  of subsystems whose dynamics and per-step cost function are coupled through
  their mean-field (empirical average). The system has mean-field sharing
  information structure, i.e., each controller observes the state of its local
  subsystem (either perfectly or with noise) and the mean-field. It is shown
  that the optimal control law is unique, linear, and identical across all
  subsystems. Moreover, the optimal gains are computed by solving two
  decoupled Riccati equations in the full observation model and by solving an
  additional filter Riccati equation in the noisy observation model. These
  Riccati equations do not depend on the number of subsystems. It is also
  shown that the optimal decentralized performance is the same as the optimal
  centralized performance. An example, motivated by smart grids, is presented
  to illustrate the result.
\end{abstract}

\section{Introduction}

\subsection{Motivation and literature overview}

Team theory investigates multi-agent decision problems in which all agents
share a common objective. Such problems arise in various applications
including cyber-physical systems, networked control systems, surveillance and
sensor networks, communication networks, smart grids, robotics, and
organizational economics. 

Not much is known regarding the optimal control of general team problems. Most
of the results in the literature are for specific information structures such
as delayed state sharing~\cite{AicardiDavoliMinciardi:1987}, delaying
observation sharing~\cite{NMT:delay-sharing}, periodic
sharing~\cite{OoiVerboutLudwigWornell:1997}, belief
sharing~\cite{Yuksel:2009}, mean-field sharing~\cite{AM:mft-cdc} and
partial-history sharing~\cite{NMT:partial-history-sharing}. We refer the
reader to~\cite{MMRY:tutorial-CDC} for an overview.

There are two main challenges in solving team problems:
\begin{enumerate}
  \item Team problems are \emph{conceptually} difficult due to the
    decentralized nature of the information available to the controllers. In
    particular, controllers need to cooperate and coordinate to minimize a
    common cost but they have different information about the state of the
    environment. Due to this discrepancy of information, dynamic programming,
    which is one of the main solution techniques for the optimal design of
    stochastic systems, does not work directly for team problems.

  \item Team problems are \emph{computationally} difficult. Even when a
    dynamic programming decomposition is obtained, the solution complexity
    increases exponentially or double exponentially with the number of
    controllers. So, it is difficult to use the dynamic programming solution
    for systems with large number of controllers.
\end{enumerate}

Similar challenges exist in dynamic games; however, they have been
successfully resolved for a class of models known as \emph{mean-field
games}~\cite{Huang:2007,Huang:2007b,LarsyLions:2007,LiZhangFeng:2008,
LarsyLions:2011,Huang:2012,Elliott:2013}.
The salient feature of
these models is that the agents/players/controllers are coupled in their
dynamics and per-step cost functions only through the mean-field (i.e., the
empirical average or the empirical distribution). If the number of players are
large, then an approximately optimal solution for these models can be obtained
by solving the infinite population limit. The infinite population limit is
easier to solve than the finite population model because, when the population
is asymptotically large, the action of a single controller does not affect the
mean-field. Therefore, the optimal solution can be obtained by solving two
coupled equations: a backward dynamic programming equation that determines the
best-response of a controller given the trajectory of the mean-field; and a
forward Fokker-Planck equation that determines the evolution of the mean-field
given the strategies of the individual controllers. A consistent solution of
these equations determine Nash equilibrium strategies. A desirable feature of
the solution is that the resultant control laws and the complexity of the
solution do not depend on the number of controllers. 

Motivated by these results, we investigate \emph{team-optimal} (rather than
Nash equilibrium) solution of mean-field models. In this paper, we focus on
linear-quadratic models. In particular, the system consists of a \emph{finite}
number of subsystems; the dynamics of the subsystems are coupled through the
mean-field of the states. The per-step cost is also coupled through the
mean-field. We assume that the controllers have \emph{mean-field sharing}
information structure introduced in~\cite{AM:mft-cdc}; that is, each
controller observes it's local state (either completely or with noise) and the
mean-field of the states of all subsystems. In some applications such as
communication networks, the sharing of the mean-field happens naturally. In
other applications, such as robotic teams, it is possible to share the
mean-field using distributed consensus algorithms.

The main difference between our setup and mean-field games is the following:
(i)~We assume a finite number of controllers and obtain team-optimal control
strategies. In mean-field games, one assumes asymptotically large number of
controllers and obtains Nash equilibrium strategies. (ii)~We assume that each
controller observes the mean-field. In mean-field games, this assumption is
not made. However, when the number of controllers goes to infinity, both the
information structures are equivalent because the mean-field becomes a
deterministic process. 

\subsection {Contributions and salient features of the results}

The linear quadratic mean-field model is a decentralized system with
non-classical information structure. \important{Our main result is to show
that linear control strategies are optimal}. The model is neither partially
nested nor quadratic invariant. In contrast, almost all positive results for
linear quadratic systems with non-classical information structures are for
models that are either partially nested or quadratic invariant. 

\important{We also show that the optimal control laws at each subsystem are
identical}. In general, identical control laws are \emph{not} optimal for
systems with exchangeable subsystems (see~\cite{AM:mft-cdc} for one an
example). So, a natural question is the following: Given a decentralized
system with exchangeable subsystems, under what conditions are optimal control
laws identical across subsystems. This question warrants further
investigation.

Our solution, and the solution complexity, do not depend on the number of
subsystems. Irrespective of the number of subsystems, the optimal control law
is given by the solution of \emph{two} (backward) Riccati equations in the
full observation model, and an additional filter Riccati equation in the noisy
observation model. The parameters of these Riccati equations do not depend on
the number of subsystems. Thus, \important{the optimal control laws can be computed
without any knowledge of the number of subsystems}.

Since the optimal control laws do not depend on the number of subsystems, our
results remain valid in the infinite population limit. In the infinite population
limit, the mean-field becomes a deterministic process that can be pre-computed
at all controllers. Thus, mean-field sharing is equivalent to the completely
decentralized information structure (where each controller only observes its
local state) considered in~\cite{Huang:2007,Huang:2007b}. Thus, \important{our
  results may be viewed as an alternative derivation of the infinite
population mean-field games solution}.

As an intermediate step of the proof, \important{we show that the
decentralized control laws---and, hence, the decentralized performance---are
same as centralized} (i.e., when all control actions are chosen by a single
controller that observes all the information available to all the
decentralized controllers). A natural follow up question that warrants further
investigation is the following: In general decentralized control problems,
when is decentralized performance the same as the centralized performance? Or
more generally, when is the performance under one information structure the
same as the performance under a finer information structure?

From an implementation point of view, the above feature has an interesting
consequence. If we have the freedom to design the information structure, then
in the linear-quadratic mean-field systems \important{there is no advantage of
sharing anything beyond the mean-field}. Since, it is easy to share the
mean-field using distributed consensus algorithms, it is relatively easy to
implement the optimal centralized solution in a distributed manner. 

Furthermore, as we already argued, in the infinite population limit,
mean-field sharing is equivalent to the completely decentralized setup (i.e.,
each controller only observes its own local state). Since, there is no
advantage of sharing any information beyond the mean-field---which is
already computable at each controller---the optimal completely
decentralized solution is also the optimal centralized solution. An immediate
corollary is that \important{the control laws obtained using the approach of mean-field
games are also the optimal centralized solution.}

\section{Problem formulation and the main results}

\subsection{Notation}

Uppercase letters denote random variables/vectors and lowercase letters denote
their realization. $\PR(\cdot)$ denotes
the probability of an event and $\EXP[\cdot]$ denotes the expectation of a
random variable. $\reals$ denotes the set of real numbers. 

For a sequence of (column) vectors $X$, $Y$, $Z$, \dots, the notation $\VVEC(X,Y,Z,\dots)$
denotes the vector $[X^\TRANS, Y^\TRANS, Z^\TRANS]^\TRANS$. The vector
$\VVEC(X_1, \dots, X_t)$ is also denoted by $X_{1:t}$. 

Superscripts index subsystems and subscripts index time. Thus, $X^i_t$ denotes
a variable at subsystem~$i$ at time~$t$. Bold letters denote collection of
variables at all subsystems. For example, $\VEC X$ denotes $(X^1, \dots,
X^n)$. $\MEAN X$ denotes the mean-field of $\VEC X$, i.e, $\MEAN
X = \frac 1n \sum_{i = 1}^n X^i$.

\subsection{Model and problem formulation} \label{sec:model}

Consider a decentralized control system with $n$ homogeneous subsystems that
operate for a finite horizon $T$. A controller is co-located with each
subsystem. Let $X^i_t \in \mathbb{R}^{d_x}$ denote the state of subsystem $i$
and $U^i_t \in \mathbb{R}^{d_u}$ denote the control action taken by
controller~$i$, $i \in \{1,\ldots,n\}$, at time $t$. Let $Z_t = \MEAN {X_t}$
denote the mean-field of the states at time~$t$.

\subsubsection{System dynamics}

The initial states $\{X^i_1\}\strut_{i = 1}^n$ are independent Gaussian random
variables with zero mean and covariance $\Sigma_X$.  They evolve in time as
follows:
\begin{equation}\label{eq:model}
  X^i_{t+1} = A_t X^i_t + B_t U^i_t + D_t Z_t + W^i_t,
  \quad  i \in \{1,\ldots,n\},
\end{equation}   
where $A_t$, $B_t$, and $D_t$ are matrices of appropriate dimensions and
$\{W^i_t\}\strut_{t = 1}^T$ is an i.i.d.\ noise process where $W^i_t$ is Gaussian with
zero mean and covariance $\Sigma_W$. We assume that the primitive random
variables $\{ X^i_1, W^i_1, \dots W^i_T \}\strut_{i = 1}^n$ are independent. 

Thus, the dynamics of the subsystems is coupled to the others through the
mean-field. 

\subsubsection{Per-step cost}
At time~$t$, the system incurs a cost that depends on the local state of the
subsystems and the mean-field as follows:
\begin{equation}
  c_t(\VEC X_t, \VEC U_t, Z_t) = 
  \frac 1n \sum_{i = 1}^n
  \Big[ \QUAD {X^i_t} {Q_t} + \QUAD {U^i_t} {R_t} \Big] + \QUAD {Z_t} {P_t} 
  \label{eq:cost}
\end{equation}
where  $Q_t$, $P_t$, $S_t$, and $R_t$ are matrices of appropriate dimension;
$Q_t$ and $P_t$ are symmetric and non-negative definite and $R_t$ is symmetric
and positive definite.

\begin{remark}
  The results derived in this paper also apply to a more general per-step cost
  of the form
  \[
    \frac 1n \sum_{i = 1}^n
    \Big[ \QUAD {X^i_t} {Q_t}  + {X^i}^\TRANS S_t Z_t + \QUAD {U^i_t} {R_t}
    \Big] + \QUAD {Z_t} {P_t}
  \]
  because it can be rewritten in the form~\eqref{eq:cost}.
\end{remark}

\subsubsection{Observation model and information structure}
We consider two observation models that differ in the observation of the local
state $X^i_t$ at controller~$i$. In the first model, called \emph{full
observation model}, controller~$i$ perfectly observes the local state $X^i_t$;
in the second model, called \emph{noisy observation model}, controller~$i$
observes a noisy version $Y^i_t \in \mathbb{R}^{d_y}$ of the local state
$X^i_t$ that is given by 
\begin{equation}\label{eq:noisy_observation}
  Y^i_t = C_t^x X^i_t + C_t^z Z_t + V^i_t, \quad  i \in \{1,\ldots,n\},
\end{equation}
where the observation noise $\{V^i_t\}_{t = 1}^T$ is i.i.d.\ Gaussian process
with zero mean and covariance $\Sigma_V$. In addition, we assume that the
noise processes are independent across subsystems and also independent of 
$\{ X^i_1, W^i_1, \dots W^i_T \}\strut_{i = 1}^n$.

In both models, in addition to the local measurement of the state of its
subsystem, each controller perfectly observes the mean-field $Z_t$.
Following~\cite{AM:mft-cdc}, we call this model \emph{mean-field sharing}
information structure.
Controllers perfectly recall all the data they observe. Thus, in the full
observation model, controller $i$ chooses control action according to
\begin{equation}\label{information structure:complete}
  U^i_t = g^i_t(X^i_{1:t},U^i_{1:t-1},Z_{1:t}),
\end{equation}
while in the noisy observation model, controller $i$ chooses control action according to 
\begin{equation}\label{information structure:incomplete}
  U^i_t = g^i_t(Y^i_{1:t},U^i_{1:t-1},Z_{1:t}).
\end{equation}
The function $g^i_t$ is called the \emph{control law} of controller~$i$. The
collection $\VEC g^i = (g^i_1,g^i_2,\ldots,g^i_T)$ is called the
\emph{control strategy} of controller~$i$. The collection
$\VEC{g} = (\VEC g^1,\ldots,\VEC g^n)$ of control strategies of all
controllers is called the \emph{control strategy of the system}.

\subsubsection{The optimization problem}
We are interested in the following optimization problem.
\begin{problem}\label{prob:main}
  In the model described above, find a joint strategy $\VEC{g} = (\VEC
  g^1,\ldots,\VEC g^n)$ that minimizes the following cost:
  \begin{equation}
    J(\VEC g) = \EXP^{\VEC g}\bigg[
      \sum_{t = 1}^T c_t(\VEC X_t, \VEC U_t, Z_t)
    \bigg],
  \end{equation}
  where the expectation is with respect to the measure induced on all system
  variables by the choice of the strategy $\VEC g$.
\end{problem}

\subsection{The main results}

\subsubsection{Full observation model}
For the full observation model, our main result is the following:

\begin{theorem}\label{Thm:main_complete}
  The optimal control  strategy of Problem~\ref{prob:main} is unique and given
  by
  \begin{equation}\label{eq:optimal_control_law_with_riccati}
    {U^i_t} = K^x_t X^i_t + (K^z_t-K^x_t) Z_t, \quad i \in \{1,\ldots,n\}.
  \end{equation}
  The optimal gains are given as follows. For the terminal time-step
  $K^x_T = K^z_T = 0_{d_u \times d_x}$ and for $t \leq T-1$,
  \begin{equation}
    K^x_t = -\left( B_t^\TRANS M^x_{t+1} B_t + R_t\right)^{-1}B_t^\TRANS M^x_{t+1} A_t,
  \end{equation}
  and 
  \begin{equation}
    K^z_t = -\left( B_t^\TRANS M^z_{t+1} B_t + R_t \right)^{-1}  B_t^\TRANS M^z_{t+1} \bar A_t
  \end{equation}
  where $\bar A_t\coloneqq A_t + D_t$ and
  $\{M^x_t\}_{t = 1}^T$ and $\{M^z_t\}_{t = 1}^T$ are solutions of the following
  two \emph{decoupled} Riccati equations:
  \begin{equation*}
    M^x_T = Q_T, \quad M^z_t = Q_T + P_T,
  \end{equation*}
  and for $t = T-1,\ldots,1$,
  \begin{align}
      M^x_t& = - A_t^\TRANS M^x_{t+1} B_t\left(B_t^\TRANS M^x_{t+1} B_t + R_t \right)^{-1} B_t^\TRANS M^x_{t+1} A_t
      \notag \\
      &\quad + A_t^\TRANS M^x_{t+1} A_t + Q_t, 
      \label{Riccati_k}
  \end{align}
  and
  \begin{align}
    M^z_t& = - {\bar A_t}^\TRANS M^z_{t+1} B_t\left(B_t^\TRANS M^z_{t+1} B_t + R_t \right)^{-1} B_t^\TRANS M^z_{t+1} \bar A_t
    \notag \\
    &\quad + \bar A_t^\TRANS M^z_{t+1} \bar A_t + Q_t + P_t. 
    \label{Riccati_s}
  \end{align}
\end{theorem}

The above result may also be derived under the weaker assumption that the
initial state $\VEC X_1 = (X^1_1, \dots, X^n_1)$ as well as the noise $\VEC
W_t = (W^1_t, \dots, W^n_t)$ are correlated across the subsystems. 

\subsubsection{Noisy observation model}
For the noisy observation model, our main result is the following:
\begin{theorem}\label{Thm:KF}
  For the noisy observation model, the optimal control strategy for
  Problem~\ref{prob:main} is unique and is given by
  \begin{equation}
    U^i_t = K^x_t \hat X^i_t + (K^z_t-K^x_t) Z_t, \quad i \in \{1,\ldots,n\},
  \end{equation}
  where optimal gains $K^x_t$ and $K^z_t$ are the same as in
  Theorem~\ref{Thm:main_complete} and  $\hat
  X^i_t = \EXP[X^i_t \mid Y^i_{1:t},Z_{1:t},U^i_{1:t-1}]$ which is generated by the
  following Kalman filtering equations: 
  \(
    \hat X^i_1 = 0
  \)
  and for $t > 1$,  
  \begin{equation}
    \hat X^i_{t+1} = A_t \hat X^i_t + B_t U^i_t + K_t^F(Y^i_t -C^x_t \hat X^i_t -C_t^z Z_t)
  \end{equation}
  where the Kalman filter gains are given by
  \begin{equation}
    K^F_t = A_t S_t {C^x_t}^\TRANS (C^x_t S_t {C^x_t}^\TRANS + \Sigma_V)^{-1},
  \end{equation}
  and the state estimation error covariances satisfy the (filter) Riccati
  equation: $S_1 = \Sigma_X$ and for $t > 1$, 
  \begin{align}
    S_{t+1} &= -A_t S_t {C^x_t}^\TRANS (C^x_t S_t {C^x_t}^\TRANS + \Sigma_V)^{-1}C^x_t S_t A_t^\TRANS
    \notag \\
    & \quad + A_t S_{t} A_t^\TRANS + \Sigma_W.
  \end{align}
\end{theorem}

\section {Proof of the main results}

The main idea of the proof is as follows. We construct an auxiliary system
whose state, control actions, and per-step cost are  equivalent to $\VEC X_t$,
$\VEC U_t$, and $c_t(\cdot)$, respectively (modulo a change of variables to be
described later). However, this auxiliary system is centrally controlled by a
single control that has access to all the information available to the
$n$~decentralized controllers in the original system. We show that the optimal
centralized solution of this auxiliary system can be implemented in the
original decentralized system, and is therefore also optimal for the
decentralized system.

\subsection{The auxiliary system}

Define $\bar X^i_t = X^i_t - Z_t$ and $\bar U^i_t = U^i_t - U^z_t$, where
$U^z_t = \MEAN {U_t}$. The auxiliary system is a centralized system with state
$\tilde X_t = \VVEC(\bar X^1_t, \dots, \bar X^n_t, Z_t)$ and actions $\tilde
U_t = \VVEC(\bar U^1_t, \dots, \bar U^n_t, U^z_t)$. Note that $\tilde X_t$ is
equivalent to $\VEC X_t$ and $\tilde U_t$ is equivalent to $\VEC U_t$.

The dynamics are the same as the model in Sec.~\ref{sec:model}. In particular, 
\begin{equation}\label{auxiliary-state-eq}
  \bar{X}^i_{t+1} = A_t\bar {X}^i_t + B_t\bar U^i_t + W^i_t - \MEAN {W_t}
\end{equation}
and
\begin{equation}\label{auxiliary-mf-eq}
  Z_{t+1} = (A_t + D_t)Z_t + B_tU^z_t + \MEAN {W_t}
\end{equation}

The per-step cost of the auxiliary model is given by $c_t(\VEC X_t, \VEC U_t,
Z_t)$ defined in~\eqref{eq:cost}.

As in Sec.~\ref{sec:model}, we consider two observation models for the
auxiliary system: full observation and noisy observation. In both cases, there
is a \emph{single centralized controller} that chooses $\tilde U_t$ based on
the observations. 

In the full observation model, the centralized controller observes $\tilde
X_t$ and chooses $\tilde U_t$ according to 
\[
 \tilde U_t = \bar g_t(\tilde X_{1:t}, \tilde U_{1:t-1}).
\]

In the noisy observation model, the centralized controller observes $\VEC Y_t
= \VVEC(Y^1_t, \dots, Y^N_t)$ where $Y^i_t$ is given by
\eqref{eq:noisy_observation} and chooses $\tilde U_t$ according to
\[
\tilde U_t = \bar g_t(\VEC Y_{1:t}, \tilde U_{1:t-1}, Z_{1:t}).
\]

In both models, we are interested in the following optimization problem:
\begin{problem}\label{prob:aux}
  In the model described above, find a strategy $\bar {\VEC g} = (\bar g_1,
  \dots, \bar g_T)$ that minimizes the following cost:
  \begin{equation}
    \bar{ J}(\bar{\VEC g}) = \EXP^{\bar{\VEC g}} 
    \Big[ \sum_{t = 1}^T c_t(\VEC X_t, \VEC U_t, Z_t) \Big],
  \end{equation}
  where the expectation is with respect to the measure induced on all system
  variables by the choice of strategy $\bar {\VEC g}$.
\end{problem}

Let $J^*$ and $\bar J^*$ denote the optimal cost for Problem~\ref{prob:main}
and Problem~\ref{prob:aux} respectively. Since the per-step cost is the same
in both cases, but Problem~\ref{prob:aux} is centralized, we have that
\(
  J^* \ge \bar J^*.
\)
We identify the optimal control laws for the auxiliary system and show that
these laws can be implemented in, and therefore are optimal for, the original
decentralized system. 

A critical step in the proof is to rewrite the cost $c_t(\VEC X_t, \VEC U_t,
Z_t)$ in terms of $\tilde X_t$ and $\tilde U_t$. For that matter, we need the
following key result.

\begin{lemma} \label{lem:key}
  For any $\VEC x = \VVEC(x^1, \dots, x^n)$ and $z = \MEAN x$, let $\bar x^i =
  x^i - z$, $i \in \{1, \dots, n\}$. Then, for any matrix $Q$ of appropriate
  dimension,
  \begin{equation}
    \frac 1n \sum_{i = 1}^n \QUAD {x^i} Q 
    =
    \frac 1n \sum_{i = 1}^n \QUAD {\bar{x}^i} Q + \QUAD z Q
  \end{equation}
\end{lemma}
\begin{proof}
  The result follows from elementary algebra and the observation that
  $\sum_{i=1}^n \bar x^i = 0$. 
\end{proof}

An immediate consequence of Lemma~\ref{lem:key} is the following:
\begin{corollary} \label{cor:cost}
  For any time~$t$, $c_t(\VEC X_t, \VEC U_t, Z_t) = \bar c_t(\tilde X_t,
  \tilde U_t)$, where 
  \begin{multline} \label{eq:aux-cost}
    \bar c_t(\tilde X_t, \tilde U_t) = 
    \frac 1n \sum_{i = 1}^n 
    \Big[
      \QUAD {\bar{X}^i_t} {Q_t} + \QUAD {\bar{U}^i_t} {R_t} 
    \Big] \\
    + \QUAD {Z_t} {(Q_t + P_t)} 
    + \QUAD {U^z_t} {R_t}.
  \end{multline}
\end{corollary}

Note that the auxiliary model has linear dynamics and in
Corollary~\ref{cor:cost} we have shown that the cost is quadratic in the state
and the control actions. Thus, the optimal control actions are linear in the
state and the corresponding optimal gains can be obtained by solving an
appropriate Riccati equation. However, the state $\tilde X_t$ of the
auxiliary system belongs to $\reals^{(n+1)d_x}$, thus, a naive attempt to
obtain an optimal solution will involve solving for $(n+1)d_x \times
(n+1)d_x$ dimensional Riccati equations. We present an alternative approach
in the next section that involves solving two $d_x \times d_x$ dimensional
Riccati equations (independent of~$n$).

\subsection {Full observation model}

The auxiliary system is a stochastic linear quadratic system. From the
certainty equivalence principle~\cite{Caines:1987}, we know that the
optimal control law is unique and identical to the control law in the
corresponding deterministic problem, whose dynamics are given by
\begin{align}
  \bar{X}^i_{t+1} &= A_t\bar {X}^i_t + B_t\bar U^i_t,
  \qquad  i \in \{1,\ldots,n\}, \\
  Z_{t+1} &= (A_t + D_t)Z_t + B_tU^z_t.
\end{align}
and the per-step cost is $\bar c(\tilde X_t, \tilde U_t)$ given
by~\eqref{eq:aux-cost}. 

Note that this system consists on $(n+1)$ components: $n$ components with
state $\bar X^i_t$ and control $\bar U^i_t$, $i \in \{1, \dots, n\}$, and one
component with state $Z_t$ and control $U^z_t$. These components have
decoupled dynamics and decoupled cost; and $n$ of these are identical. Thus,
the optimal control law of each component may be identified separately.
Therefore, we have the following:
\begin{theorem}\label{Thm:main_complete-auxiliary}
  The optimal control strategy of auxiliary model is unique and given by
  \begin{equation}\label{eq:optimal_control_law_with_riccati_aux}
    \bar U^i_t = K^x_t \bar X^i_t, \quad U^z_t = K^z_t Z_t,\quad i \in \{1,\ldots,N\},
  \end{equation}
  where the gains $K^x_t$ and $K^z_t$ are given as in
  Theorem~\ref{Thm:main_complete}.
\end{theorem}

To complete the proof of Theorem~\ref{Thm:main_complete}, note that
\[
  U^i_t = \bar U^i_t + U^z_t = K^x(X^i_t - Z_t) + K^z Z_t.
\]
Thus, the control laws specified in Theorem~\ref{Thm:main_complete} are the
optimal \emph{centralized} control laws, and, a fortrori, the optimal
decentralized control laws. 

\subsection {Noisy observation model}

\begin{figure*}
  \centering
  \includegraphics[width=0.8\textwidth]{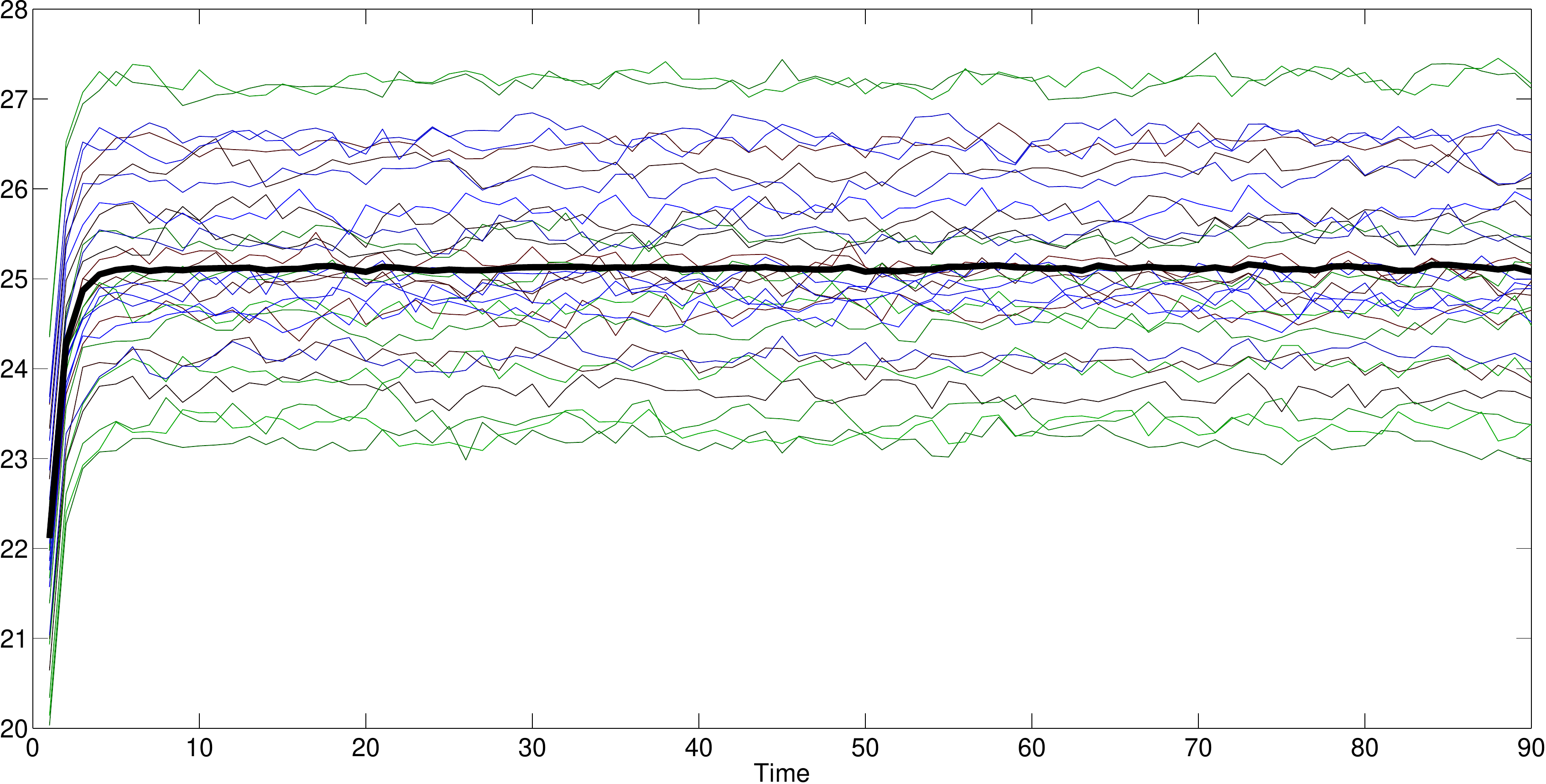}
  \caption{An example of mean-field control of space heaters. The light lines
  show the temperature of individual heaters and the thick line depicts the
mean-field of the trajectories. Note that the individual trajectories adjust
around their initial values so that the mean-field tracks the reference
trajectory $z_{ref} = 25$. The tracking is not perfect because of the
friction between tracking the initial states locally and tracking the
mean-field globally.}
  \label{fig:plot}
\end{figure*}

Define
\begin{align*}
  \mathring X^i_t  &= \EXP[ \bar X^i_t \mid \VEC Y_{1:t}, \tilde U_{1:t-1}, Z_{1:t} ],
  \\
  \breve X^i_t &= \EXP[ X^i_t \mid \VEC Y_{1:t}, \tilde U_{1:t-1}, Z_{1:t} ].
\end{align*}
Since $X^i_t = \bar X^i_t + Z_t$, we have that $\breve X^i_t = \mathring X^i_t
+ Z_t$. 

From the Separation Theorem for linear quadratic Gaussian
systems~\cite{Caines:1987}, we know that the optimal (centralized)
control law for the auxiliary system is given by
\[
  \bar U^i_t = K^x_t \mathring X^i_t, \quad U^z_t = K^z_t Z_t,\quad i \in \{1,\ldots,N\}.
\]
Or equivalently, 
\[
  U^i_t = K^x_t \breve X^i_t + (K^z_t-K^x_t) Z_t, \quad i \in \{1,\ldots,n\}.
\]
Thus, to prove Theorem~\ref{Thm:KF}, we need to show that $\breve X^i_t = \hat
X^i_t$ (where $\hat X^i_t$ is defined in Theorem~\ref{Thm:KF}). This is proved
below.
\begin{lemma}
   In the auxiliary system with noisy observation, $\breve X^i_t = \hat
   X^i_t$.
\end{lemma}
\begin{proof}
  Recall that $X^i_t$ evolves according to~\eqref{eq:model} and $Y^i_t$ is
  given by~\eqref{eq:noisy_observation}. Therefore, for any $i,j \in \{1,
  \dots, n\}$ such that $i \neq j$, $(X^i_{1:t}, Y^i_{1:t})$ is conditionally
  independent of $(X^j_{1:t}, Y^j_{1:t})$. Consequently,
  \begin{align*}
    \breve X^i_t &= \EXP[ X^i_t \mid \VEC Y_{1:t}, \tilde U_{1:t-1}, Z_{1:t} ]
    \\
    &= \EXP[ X^i_t \mid Y^i_{1:t}, \tilde U_{1:t-1}, Z_{1:t} ] \\
    &= \EXP[ X^i_t \mid Y^i_{1:t}, \VEC U_{1:t-1}, Z_{1:t} ] \\
    &\stackrel{(a)}= \EXP[ X^i_t \mid Y^i_{1:t}, U^i_{1:t-1}, Z_{1:t} ] \\
    &= \hat X^i_t,
  \end{align*}
  where $(a)$ follows because $X^i_t$ depends on $\VEC U_{1:t-1}$ only through
  $U^i_{1:t-1}$. 
    
\end{proof}

\section {Numerical Example}

We simulate an example of inspired by the  \emph{collective target tracking
mean-field model} of~\cite{KizilkaleMalhame:2013}. We consider a population
of space heaters. The state $X^i_t$ denotes the room temperature at
heater~$i$ at time~$t$. We assume that the inital mean temperature is $z_1$
degrees, with individual temperatures distributed according to Gaussian
distribution with unit variance around $z_1$ degrees. We assume that the
temperature dynamics are linearized around the operating point and are given
by
\begin{equation*}
  X^i_{t+1} = a(X^i_t - x_0) + b U^i_t + W^i_t
\end{equation*}
where $x_0$ is the ambient temperature, $W^i_t$ is the randomness due to the
environment and $U^i_t$ is the control action of a local controller. 

It is desired that the mean temperature increases to 21 degrees (denoted by
$z_{ref}$) in $T$ time steps (we assume $T = 90$ minutes). The per-step cost
function is
\[
  \frac 1n \sum_{i = 1}^n 
  \Big[
    q(X^i_t - X^i_1)^2 + r {U^i_t}^2 
  \Big]
  + p (Z_t - z_{ref})^2
\]
The rationale of this cost function is as follows. It is assumed that the
inital temperature is the comfort level of user~$i$; so we penalize
local deviations from the initial temperature. The second term corresponds
to the local control energy. The last term is the penalty for the mean-field
deviating from the reference mean-field. The objective is to minimize the
expected total cost over a finite horizon.

The above model is a tracking problem. The results of
Theorem~\ref{Thm:main_complete} extend to tracking problems in a natural
manner. We solve the Riccati equations for the following values of the
parameters:
\begin{align*}
  n &= 30,  & a &= 0.8, & b &= 1, \\
  q &= 0.5, & p &= 1, & r &= 1, \\
  x_0 &= 22, & z_1 &= 22, & z_{ref} &= 25, \\
  T &= 90, & W^i_t &\sim \mathcal N(0,1), & X^i_1 &\sim \mathcal N(22, 2).
\end{align*}
The resultant trajectories are shown in Fig.~\ref{fig:plot}.

\section {Conclusion}

We consider team-optimal control of finite number of mean-field coupled LQG
subsystems with mean-field sharing information structure. The optimal control
law is unique, identical for all controllers, and linear in the current state
and the current mean-field; the optimal gains are obtained by solving two
decoupled Riccati equations for the full observation case and an additional
filter Riccati equation for the noisy observation case. These Riccati
equations do not depend on the number of subsystems. 

To prove the main results, we consider an auxiliary centralized system whose
dynamics and cost are equivalent to the original decentralized system. It is
shown that the optimal centralized control law are implementable with
mean-field sharing. Thus, the optimal decentralized control laws are
identical to the optimal centralized control laws.

Although we only presented the results for the finite horizon optimal
regulation problem, the results extend to infinte horizon and to optimal
tracking problems in a natural manner. Moreover, it is possible to extend the
results to non-homogeneous population consisting of multiple types.

\bibliographystyle{IEEEtran}
\bibliography{IEEEabrv,../../collection,../../personal}

\end{document}